\documentclass[
  a4paper, 
  reqno, 
  oneside, 
  11pt
]{amsart}

\usepackage[dvipsnames]{xcolor} 
\colorlet{cite}{LimeGreen!50!Green}
\usepackage{tikz}  
\usetikzlibrary{arrows,positioning}
\tikzset{ 
  baseline=-2.3pt,
  text height=1.5ex, text depth=0.25ex,
  >=stealth,
  node distance=2cm,
  mid/.style={fill=white,inner sep=2.5pt},
}

\usepackage{lmodern}
\usepackage{tikz-cd}
\usepackage{amsthm, amssymb, amsfonts}
\usepackage{soul}

\usepackage{graphicx,caption,subcaption}
\usepackage{braket}  

\usepackage[%
  bookmarks=true,			
  unicode=true,			
  pdftitle={Kuranishi map},		%
  pdfauthor={Elizabeth Gasparim}{Francisco Rubilar},	%
  pdfkeywords={Vector bundles}{Kuranishi map}{Surfaces}{general type},	
  colorlinks=true,		
  linkcolor=Blue,			
  citecolor=cite,		
  filecolor=magenta,		
  urlcolor=RoyalBlue			
]{hyperref}				
\usepackage{cleveref}

\newtheoremstyle{mydef}
  {}		
  {}		
  {}		
  {}		
  {\scshape}	
  {. }		
  { }		
  {\thmname{#1}\thmnumber{ #2}\thmnote{ #3}}	

\theoremstyle{plain}	
\newtheorem{theorem}{Theorem} 
\newtheorem{lemma}[theorem]{Lemma} 
\newtheorem{proposition}[theorem]{Proposition}
\newtheorem*{theorem*}{Theorem}

\theoremstyle{mydef} 
\newtheorem{definition}[theorem]{Definition}
\newtheorem*{conjecture*}{Conjecture}
\theoremstyle{remark}
\newtheorem{remark}[theorem]{Remark}
\newtheorem{notation}[theorem]{Notation}

\DeclareMathOperator{\im}{im}

\DeclareMathOperator{\HH}{\mathrm{H}}

\DeclareMathOperator{\rank}{rank}
 
\DeclareMathOperator{\End}{End}
\DeclareMathOperator{\Ext}{Ext}
\DeclareMathOperator{\Hom}{Hom}

\newtheorem*{proposition*}{Proposition}
\newtheorem*{lemma*}{Lemma}
\newtheorem*{corollary*}{Corollary}
\theoremstyle{definition}
\theoremstyle{remark}
\newtheorem*{claim*}{Claim}

\DeclareMathOperator{\Pic}{Pic}

\newcommand{\ce}{\mathrel{\mathop:}=}

\usepackage{amsmath,amssymb,amsthm,mathrsfs,calligra,amsfonts}
\newcommand*{\shom}{\mathscr{H}\text{\kern -2pt {\it{om}}}\,}
\newcommand*{\ext}{\mathscr{E}\text{\kern -1pt {\it{xt}}}\,}

\author{E. Ballico, E. Gasparim, F. Rubilar, B. Suzuki}
\address{EB: Dept. Mathematics, University of Trento, I-38050 Povo, Italy.\newline
EG, FR, BS:	Depto. Matem\'aticas, Universidad Cat\'olica del Norte,  Chile. \newline
{emails: ballico@science.unitn.it, etgasparim@gmail.com, francisco.rubilar@ucn.cl, obrunosuzuki@gmail.com}}
\thanks{Corresponding author: B. Suzuki, obrunosuzuki@gmail.com}
\title[Kuranishi map  for bundles on  certain products of curves]{The Kuranishi map for  vector bundles  \\on certain products of curves}

\begin{document}

\maketitle 

\begin{abstract}
We describe  deformations  of vector bundles on  surfaces that are 
a product of two  smooth projective curves.  We explicitly describe the Kuranishi map around unstable 
vector bundles  and  compare the homologies of the Kuranishi spaces of stable and unstable deformations. 

\end{abstract}
\vspace{5mm}

\noindent Keywords: Deformation theory, Kuranishi map, stable bundle,  homology,
 Kuranishi deformation space.

\vspace{3mm}

\noindent MSC: Primary  32G08; Secondary  32G05, 14J60.

\section{Moduli}

 We study   deformation spaces of vector bundles over  complex surfaces   that are 
products of  two  smooth projective curves. We take the  product 
$\Sigma\ce \Sigma_1\times \Sigma_2$ of general curves of genera
 $1< g_1 <g_2$, so that   the surface $\Sigma$  is of general type.
Kuranishi deformations of bundles on  such  surfaces are particularly manageable 
 because they can be expressed in terms of  deformations  of bundles on each of the individual curves.
In fact, we assume throughout  that $\Pic(\Sigma)$ is generated by the box products of 
pullbacks of line bundles from $\Sigma_i$. 
This property, called Pic-independence  \cite{F}, 
 holds for a very general choice of $(X,Y)$  in 
$\mathcal M_{g_1} \times \mathcal M_{g_2}$,
i.e. for all $(X,Y)$ outside countably many proper subvarieties of $\mathcal{M}_{g_1}\times \mathcal{M} _{g_2}$, 
see \cite[Prop.\thinspace 3.19]{I} and \cite[Th.\thinspace 11.5.1]{BL}.

Fix the smooth type of a rank 2 vector bundle $\mathcal E$ on $\Sigma_1\times \Sigma_2$, or equivalently, its
 Chern classes. Here we will consider only the case of trivial determinant, hence $c_1(E)=0$. 
 For each holomorphic bundle $E$ with
 the smooth type of $\mathcal E$, we discuss the deformation space of $E$. 
 
The Kodaira--Spencer theory of deformations implies that any  Zariski
	tangent vector of the local moduli space of holomorphic vector bundles at $E$
	can be
	interpreted as an element of $\HH^1(\Sigma,\End E)$, where $\End E$ is the 
	endomorphism bundle of $E$, see  \cite{CWY}.
The Kuranishi map associated to $E$ can be written as
$$
\Psi_{E}\colon \HH^1(\Sigma,\End E) \rightarrow \HH^2(\Sigma,\End E )
$$
$$w \mapsto [w,w]$$
and has the following properties \cite[Prop.\thinspace 6.4.3]{DK}:

 (i) $\Psi_E$ and its derivative both
vanish at $0,$ and a versal deformation
of $E$ is parametrized by the complex space
$\Psi^{-1}_E(0)$.

 (ii) The two-jet of $\Psi_E$ at the origin is given by 
the combination of cup product and bracket:
\begin{equation*}
	 \HH^1(\End E) \otimes \HH^1(\End E)
 \rightarrow\HH^2(\End_0 E).
\end{equation*}
\label{DKprop}
\begin{definition}\label{def}
We call the germ of $\Psi_{E}^{-1}(0)$ the {\bf Kuranishi deformation space} of $E$ 
 and denote it by  ${\mathbb K}(\Sigma, E) $. 
We denote by  ${\mathbb K}^s(\Sigma, E) $
 the subset of ${\mathbb K}(\Sigma, E) $ consisting of  stable rank 2 bundles.
 \end{definition}

\noindent We will omit the first entry $\Sigma$ when it is clear from the context, 
and  add a subscript ${\mathbb K}_{c_2}(\Sigma, E) $ and  ${\mathbb K}^s_{c_2}(\Sigma, E) $ respectively
when we want to fix the value of the second Chern class.

\begin{remark}For a fixed choice of $E$ it may happen that ${\mathbb K}^s(\Sigma, E) $ is empty (see App.\thinspace\ref{unstable}). 
In such a case we will see in Secs.\thinspace\ref{Kur-split} and \ref{nonsplit} that the Kuranishi map 
reduces to a single component, labelled $\iota)$ in both cases. Furthermore, in such a case, 
it then also happens
that the target vanishes, so that the Kuranishi deformation space of such a bundle is 
the entire vector space $\HH^1(\End E)$. 
%
\end{remark}

We prove:

\begin{theorem*}[\ref{main}] Let $\Sigma_1$ and $\Sigma_2$ be Pic-independent smooth projective curves.
Let $E$ be a bundle on $\Sigma_1\times \Sigma_2$ with trivial determinant
and with second Chern number  $c_2> 8g_1g_2$. 
Then, for $0<q<c_2$
we have 
$$H_q({\mathbb K}_{c_2}
(\Sigma, E), {\mathbb K}^s_{c_2}(\Sigma, E))=0.$$
\end{theorem*}

\section{Background}

We recall the basic concept of stability  used here. 
\begin{definition}\label{degreeofL}
	Let $X$ be a compact K\"ahler complex manifold and $E \to X$ a holomorphic  vector bundle. Then
	 \label{deg2} if $\dim{X} = n > 1$, the \textit{degree} of $E$ is   
		\[
		\deg{E}=\int_X c_1(E) \wedge \omega^{n-1},
		\]
		where $\omega \in\HH^2(X, \mathbb{Z})$. The form $\omega$ is called a \textit{polarization} of $E$. 
	The definition of degree depends on the chosen polarization.
	The \textit{slope} of $E $ is
	$
	\mu(E) = \deg(E)/\rank(E).
	$
	 Then $E$ is
 \textit{slope stable} (resp. \textit{semistable}) if $\mu(E') < \mu(E)$ (resp. $\mu(E') \leq \mu(E)$) for every proper subbundle $0 \to E' \to E.$
\end{definition}

The most natural choice for the product of curves is to take a polarization such that 
$\omega$ is a class of type $(1,1)$, making the contributions of $\Sigma_1$ and $\Sigma_2$ 
well balanced.

\begin{remark} 
Moduli spaces of vector bundles  correspond  to moduli spaces of instantons 
via the 
celebrated
 Kobayashi--Hitchin correspondence.
In 1978  Atiyah and Jones  \cite{AJ}  conjectured that 
the inclusion of the moduli space of instantons 
of charge $k$ into the moduli space of all connections 
induces isomorphisms  in
homology and homotopy  in  degrees less than $k/2$.
 The conjecture was proved for $\mathrm{SU}(2)$ instantons on $S^4$   \cite{BHMM},
 for $\mathrm{SU}(n)$ ins\-tantons $S^4$  \cite{T}, 
  for ruled surfaces  \cite{HM},  and  for rational surfaces   \cite{Ga};
it  remains open in all other cases.
Other results on the stable topology of moduli spaces include \cite{Ta,CW, Sa}.
Our result here contributes to this type of topological  questions, by making  the technical step of verifying 
the topological implications of any particular choice of stability unnecessary.
\end{remark}

We will apply Kirwan's techniques for removing subvarieties of high 
codimension,  while preserving isomorphisms in homology through a range.

In \cite[Cor.\thinspace 6.4]{KI} Kirwan  proved the following result for a
 quasi-projective variety $X$. Let  $\mu$
be a non-negative integer such that every $x_0 \in X$ has a 
neighbourhood in $X$ which is analytically isomorphic to an open subset of 
$$\{x \in \mathbb{C}^N | f_1(x)= \cdots = f_M(x)=0\}$$
for some integers $N,$ $M,$ and holomorphic functions 
$f_i$ depending on $x_0$ with $M \leq \mu.$ 
If $Y$ is a closed subvariety of codimension
$k$  in $X,$ then for $ q < k-\mu,$
\begin{equation}\HH_q(X, X-Y) =0= \HH^q(X,X-Y).
\label{kirwan}\end{equation}
The same result applies, with an identical proof, with $X$ a complex 
space instead of a quasi-projective variety, and we shall use the complex
version in our calculations.

\section{The Kuranishi map around a split bundle }
\label{Kur-split}

 Consider the surface
 $\Sigma= \Sigma_1 \times \Sigma_2$  and let $L = L_1 \boxtimes L_2$ be a line bundle over $\Sigma$ of type $(m,n)$.
In this section we study the Kuranishi map near a 
split bundle $E= L \oplus L^{-1}.$ We carry out the study for the case of $c _1=0$, other cases are just notationally 
heavier, but do not bring any real extra difficulty. 
 Our aim is to describe explicitly points whose Kuranishi deformation space that have obstructed deformation 
theory.  At such points we investigate the nature of the 
singularity in comparison to the dimension of the set of stable bundles.

Given the splitting $E =L \oplus L^{-1}$, we have that the Kuranishi space at $E$
is contained in 
\[ \HH^1(\Sigma, \End E)= \HH^1(\Sigma, L^2) \oplus \HH^1(\Sigma,\mathcal{O}) \oplus \HH^1(\Sigma,\mathcal{O})  \oplus \HH^1(\Sigma,L^{-2}).\]
We identify the directions as:

\begin{notation} $T_u= \HH^1(\Sigma, L^2),$ 
$T_o =  \HH^1(\Sigma,{\mathcal O})$ and 
$T_s = \HH^1(\Sigma,L^{-2}).$
\end{notation}
Observe  that $T_u $ and $T_o$ contain unstable deformations 
of $E$ whereas $T_s$ is the direction of stable deformations. 
We consider the components of the Kuranishi map at $E$:  

\begin{enumerate}
	\item[ $\iota)$] $ \HH^1(\Sigma,L^2) \otimes \HH^1(\Sigma,\mathcal{O})  \to \HH^2(\Sigma,L^2)$,
	\item[$\iota\iota)$] $\HH^1(\Sigma,L^2)\otimes \HH^1(\Sigma,L^{-2}) \to \HH^2(\Sigma,\mathcal{O})$, 
	\item[$\iota\iota\iota)$] $\HH^1(\Sigma,\mathcal{O})\otimes \HH^1(\Sigma,L^{-2}) \to \HH^2(\Sigma,L^{-2})$.
\end{enumerate}

We could also  list 
  the component $\HH^1(\Sigma, \mathcal{O}) \otimes \HH^1(\Sigma, \mathcal{O}) \to \HH^2(\Sigma, \mathcal{O})$, but it  is not needed
for our analysis, since  $\HH^2(\Sigma, \mathcal{O})$ poses a bounded number of obstructions and
 $ \HH^1(\Sigma, \mathcal{O})$ does not depend on the choice of the point $E$ 
and accordingly  does not play an interesting role  for this task.
We will use the local study  to estimate the codimension of the set of unstable bundles.
This section serves as a warm-up for the more general calculation that will be carried out
in Sec.\thinspace\ref{Kur-nonsplit} around nonfiltrable bundles.

Observe that the Kuranishi map around any bundle $E$ will have another  component with target
space $\HH^2(\Sigma, \mathcal O)$. We may regard equations 
coming from this component as unavoidable equations. Thus, in a sense these re\-present a minimal number
of equations that would be present overall in case we also  allowed the complex structure of the base space to vary. 
Our choice here is to fix the complex structure of the surface and vary only that of  vector bundles,
disregarding this component. 

\begin{proposition}\label{split}Assume $c_2>8g_1g_2$.
Then removing the set of split points does not change the homology of the Kuranishi space 
$\mathbb K_{c_2}(\Sigma, E)$
in dimension smaller than $c_2$. 
\end{proposition}
	
\begin{proof} Recall that in this case $c_2= -2mn$. We assume that $c_2>8g_1g_2$.

 $\iota)$ We claim that  $\HH^2(\Sigma,L^2)=0$
	for large $m$ and hence the first component can 
	be ignored, since it will produce no equations for the local model at $E$. In fact, 
$\HH^2(\Sigma,L^2)=$  
	$$ \HH^2(\Sigma_1,L_1) \otimes \HH^0(\Sigma_2,L_2)
	\oplus \HH^1(\Sigma_1,L_1) \otimes \HH^1(\Sigma_2,L_2) 
	\oplus \HH^0(\Sigma_1,L_1) \otimes \HH^2(\Sigma_2,L_2) 
	$$
	and $\HH^2(\Sigma_1, L_1)=\HH^2(\Sigma_2, L_2)=0$. 
	Furthermore, if $m=c_1(L) > 2g(X)-2$,
	then 
	$
	\HH^1(X, L) = 0$,
producing no obstructions for large $c_2$.
\vspace{3mm}	

$\iota\iota)$ Here, ${\mathcal O}$ is fixed by the choice of $\Sigma$ and this is 
	precisely the part that poses $h^2(\Sigma, \mathcal O)$ obstructions.
	These are unavoidable, but are a bounded number depending only on the topology of 
	$\Sigma$ and small once we compare to the deformation space  with large $c_2$.
\vspace{3mm}
	
 $\iota\iota\iota)$ This component $T_o \otimes T_s \rightarrow \HH^2(\Sigma,L^{-2})$  can  give a number of obstructions that grows with $m$.
Here we wish to remove the 
	set $T_o$ which contains the directions of unstable deformations,  in the case when obstructions are present. 
	Since
	\begin{align*}
	T_s =& \left( \HH^1(\Sigma_1, L_1^{-2}) \otimes \HH^0(\Sigma_2,L_2^{-2}) \right) \oplus \left( 
	\HH^0(\Sigma_1, L_1^{-2}) \otimes \HH^1(\Sigma_2,L_2^{-2}) \right), 	\\
	T_o =& \left( \HH^0(\Sigma_1,{\mathcal O}) \otimes \HH^1(\Sigma_2, {\mathcal O}) \right) \oplus \left( 
	\HH^1(\Sigma_1,{\mathcal O}) \otimes \HH^0(\Sigma_2, {\mathcal O}) \right),
	\end{align*}
	and $\HH^0(\Sigma_1,{\mathcal O}) = \HH^0(\Sigma_2, {\mathcal O}) = \mathbb{ C}\,\, \textup{and}
	\HH^0(\Sigma_1,L_1^{-2})=0\textup{ for large $m$,}$ we have 
$$
		T_s  = \HH^1(\Sigma_1,L_1^{-2}) \otimes \HH^0(\Sigma_2,L_2^{-2}),\quad \text{and}\quad
		T_o  = \HH^1(\Sigma_1,\mathcal{O}) \oplus \HH^1(\Sigma_2,\mathcal{O}).
$$
	Therefore 
	\begin{align*}
		T_o \otimes T_s 
		=& \left( \HH^1(\Sigma_1,\mathcal{O}) \oplus \HH^1(\Sigma_2,\mathcal{O})\right) \otimes \left(\HH^1(\Sigma_1,L_1^{-2}) \otimes \HH^0(\Sigma_2,L_2^{-2}) \right) \\
		=& \HH^1(\Sigma_1,\mathcal{O}) \otimes \HH^1(\Sigma_1,L_1^{-2}) \otimes \HH^0(\Sigma_2,L_2^{-2})  \\
		& \oplus 
		\HH^1(\Sigma_2,\mathcal{O}) \otimes \HH^1(\Sigma_1,L_1^{-2}) \otimes \HH^0(\Sigma_2,L_2^{-2})\\
		\simeq 
		& \HH^1(\Sigma_1,L_1^{-2}) \otimes  \HH^1(\Sigma_1,\mathcal{O}) \otimes \HH^0(\Sigma_2,L_2^{-2})  \\
		& \oplus 
		\HH^1(\Sigma_1,L_1^{-2}) \otimes\HH^1(\Sigma_2,\mathcal{O}) \otimes  \HH^0(\Sigma_2,L_2^{-2})
		,
	\end{align*}
	where we may change the order, since the product is taken over $\mathbb C$.\\
	Furthermore, 
	$
	\HH^2(\Sigma,L^{-2}) = \HH^1(\Sigma_1,L_1^{-2}) \otimes \HH^1(\Sigma_2, L_2^{-2})$,	so the obstruction map can be written as 	
	\begin{multline}
		\left(\HH^1(\Sigma_1, L_1^{-2}) \otimes \HH^0(\Sigma_2,L_2^{-2})  
		\otimes \HH^1(\Sigma_1,{\mathcal O}) \right)
		\oplus \\
		\left( \HH^1(\Sigma_1, L_1^{-2}) \otimes \HH^0(\Sigma_2,L_2^{-2})  
		\otimes \HH^1(\Sigma_2,{\mathcal O}) \right)\\
		\rightarrow
		\HH^1(\Sigma_1, L_1^{-2}) \otimes \HH^1(\Sigma_2,L_2^{-2}) .\label{obsmap}
	\end{multline}

	The first coordinates map by the identity, therefore it suffices to study the 
	zero set of the map 
	$$ \left( \HH^0(\Sigma_2,L_2^{-2})  
	\otimes \HH^1(\Sigma_1,{\mathcal O}) \right)
	\oplus \left(
	\HH^0(\Sigma_2,L_2^{-2})  
	\otimes \HH^1(\Sigma_2,{\mathcal O}) \right)
	\rightarrow
	\HH^1(\Sigma_2,L_2^{-2}) , $$
	that is,
	\begin{equation}
	 \HH^0(\Sigma_2,L_2^{-2})  
	\otimes \left(\HH^1(\Sigma_1,{\mathcal O}) 
	\oplus 
	\HH^1(\Sigma_2,{\mathcal O}) \right)
	\rightarrow
	\HH^1(\Sigma_2,L_2^{-2}).\label{zeroset}
	\end{equation}	

	\begin{claim*}
			The topology of $\mathbb K_{c_2}(\Sigma, E) \cap (T_o \oplus T_s)$ remains unchanged upon removing $T_o$ up to real dimension $c_2$.
\end{claim*}
	 
	We set $V_1(m)\ce \HH^1(\Sigma_1,L_1^{-2}),
	V_0(n) \ce \HH^0(\Sigma_2,L_2^{-2})$ and  $$A\ce T_{o} = \HH^1(\Sigma_1, {\mathcal O}) \oplus \HH^1(\Sigma_2, {\mathcal O})$$
	and we want to identify the zero set  of the Kuranishi map (\ref{obsmap}). Setting $K = V_0(n) \otimes A$ as the zero set of the quadratic  cup-product
	map (\ref{zeroset}), the zero set of the Kuranishi map (\ref{obsmap}) in 
	$V_1(m) \otimes V_0(n) \otimes A$ becomes 
	\[
	K^{\nu_1}= K \times_{A} K \times_A \cdots \times_A K 
	\]
	$\nu_1=\dim_{\mathbb{C}}{V_1(m)}$  times.
	
	Let $K_a$ be the fiber over a point $a\in  A.$ We want to 
	remove $A,$ that is, the zero section of the fiber product $K.$ 

	Note that	if $K_a = 0$ for some $a,$ then $T_o$ is an irreducible component
		of $\mathbb K_{c_2}(\Sigma, E) \cap  (T_o \oplus T_s)$
		and in that case $\mathbb K_{c_2}(\Sigma, E)$ has a totally unstable component.

	We have avoided this case 
	by demanding in our hypothesis 
	that we work in the complement of unstable components. Therefore, in our case
	$K_a$  is never zero.

   Now, observe that in Eq.\thinspace(\ref{zeroset}) we have that $T_o$ appears in codimension 
    $h^0(\Sigma_2,L_2^{-2}) $ and the target gives us $h^1(\Sigma_2,L_2^{-2})$ equations. So, to have a meaningful bound here, 
all we need is that the number of equations be less than the codimension. But, by Riemann--Roch we have:
 $$h^0(\Sigma_2,L_2^{-2}) - h^1(\Sigma_2,L_2^{-2})=  -2n - g_2 + 1>0.$$
 
 Since $K_a$ is never zero, we can then promote this to the bounds we require in Eq.\thinspace (\ref{obsmap}), using the fact that 
  $\nu_1 = 2m-g_1+1.$ So, we have that $T_o$ appears in the Kuranishi map (\ref{obsmap}) in codimension 
   $\nu_1  \mbox{dim}K_a=\nu_1\nu_0=(2m-g_1+1)(-2n-g_2+1)$ and the corresponding neighbourhood 
   is defined by $h^1(\Sigma_1, L_1^{-2}) h^1(\Sigma_2,L_2^{-2})$ equations.  
   
 Thus, we obtained the following bound on codimension of $T_o$ minus number of defining equations, to apply estimate (\ref{kirwan}):
 $$h^1(\Sigma_1, L_1^{-2})(h^0(\Sigma_2,L_2^{-2}) - h^1(\Sigma_2,L_2^{-2}))=  (2m-g_1+1) (-2n - g_2 + 1)>c_2(E)$$
whenever $m>>0$ is large enough to make $h^0(\Sigma_1, L_1^{-2})= 0$.

Note that the last inequality holds when both $2m-g_1+1>0$  and $-2n - g_2 + 1>0$, hence
$c_2= -2mn >2(2g_1-2)(2g_2-2)$. Thus, it suffices to require $c_2 >8g_1g_2.$
\end{proof}

\section{The Kuranishi map around a nonfiltrable bundle}\label{Kur-nonsplit}


We work with a general surface $\Sigma=  \Sigma_1 \times \Sigma_2$    with    $1<g_1<g_2$.
In such a case $\Sigma$ is a minimal surface of general type with irregularity $q=g_1+g_2$, 
and geometric genus $p_g=g_1g_2$, $K^2=4(g_1-1)(g_2-1)$.

 Here $L = L_1 \boxtimes L_2$ is a line bundle over $\Sigma$ of type $(m,n)$ over 
 $\Sigma_1 \times \Sigma_2$.
We now study the Kuranishi map near an unstable  bundle,
in the most general (and most frequent) case,  namely,
 when $E$ 
is not an extension of line bundles, instead when
$E$ is only an extension of a rank one torsion free sheaf $F$ 
 by a destabilizing line bundle $L$. See  \cite[Ch.\thinspace 2, Prop.\thinspace 5] {Fr} for the proof 
 that over any smooth 
 projective surface, a rank 2 bundle $E$ can be written as such an extension.
 Hence
$$0 \rightarrow L \rightarrow E \rightarrow F \rightarrow 0$$
and 
$$0 \rightarrow F \rightarrow F^{\vee\vee} \rightarrow Q \rightarrow 0$$
with $Q$ supported at points and  $F^{\vee\vee} = L^{-1}.$
The tangent space to $\mathfrak{gl}_2$  deformations of $E$ is 
$$ \Ext^1(L,L) \oplus \Ext^1(L,F) \oplus \Ext^1(F,L) \oplus 
\Ext^1(F,F) $$
and there is a map
$$\mathbb K(E) \hookrightarrow \Ext^1(L,L) \oplus \Ext^1(L,F) \oplus \Ext^1(F,L) \oplus 
\Ext^1(F,F) .$$
 Since 
$\Ext^1(L,L) = \HH^1(\mathcal{O})$ the space of $\mathfrak{sl}_2$ deformations 
of $E$ is 
$$  \Ext^1(L,F) \oplus \Ext^1(F,L) \oplus 
\Ext^1(F,F). $$
Considering only $\mathfrak{sl}_2$ deformations, we may write the map as 
$$\mathbb K(E) \hookrightarrow T_s \oplus T_u \oplus T_o$$
 where we set the 
{\bf notation}:

 $T_s\colon =\Ext^1(L,F) = \HH^1(L^{-1}\otimes F),$ 
 
 $T_u\colon = \Ext^1(F,L)$

 $T_o\colon = \Ext^1(F,F) =\HH^1(\mathcal{O}) \oplus \Gamma \ext^1(F,F).$

\noindent

Note that $T_s$ contains deformations towards stable bundles, whereas both 
$T_u$ and $T_o$ give directions of unstable deformations. $T_o$ is the 
direction corresponding to varying the holomorphic structure of $L$ and
 a lengthy but straightforward calculation shows that its splitting is given by 
\[
\mathrm{Ext}^1(F,F) = \Gamma \ext^1(F,F)
\oplus \HH^1(\mathcal{O}).
\]

 The components of the Kuranishi map are
 \begin{eqnarray*}\label{K2}
\left.\iota\right) \Ext^1(F,L) \otimes \Ext^1(F,F) 
&  \rightarrow & \Ext^2(F,L)\label{7}\\
\left.\iota\iota\right) \Ext^1(L,F) \otimes \Ext^1(F,L)
&  \rightarrow & \Ext^2(F,F)= \HH^2(\mathcal{O})\label{8}\\
\left.\iota\iota\iota\right) \Ext^1(F,F) \otimes \Ext^1(L,F) 
&  \rightarrow & \Ext^2(L,F)= \HH^2(L^{-1}\otimes F)=\HH^2(L^{-2}).\label{9}
 \end{eqnarray*}
In this case we have that $c_2=-mn+l(Q)$ and we proceed to discuss the topology 
when $c_2$ goes to infinity which may happen if 
either $-mn$ or  $l(Q)$ goes to infinity. We will see 
that following  the argument takes us to the same maps and bounds 
obtained in Sec.\thinspace \ref{Kur-split} in either case. 
We now study each component of the Kuranishi map.\\

$\left.\iota\right)$ Consider the map 
$ \Ext^1(F,L)  \otimes  \Ext^1(F,F)\rightarrow \Ext^2(F,L).$

We will use the fact  $\Ext^1(F,F)=\HH^1(\mathcal{O})\oplus\Gamma \ext^1(F,F)$, 
proved in Lem.\thinspace \ref{locglob} of App.\thinspace  \ref{extcalc}; together with the fact that
$\Gamma \ext^1(F,F)= \Gamma \shom(F,Q)= \Hom(F,Q)$
maps to $\Ext^1(F,F)$ since we have the short exact sequence 
$$0 \rightarrow F \rightarrow F^{\vee\vee} \rightarrow Q \rightarrow 0.$$
Taking $\Hom(., L )$, we obtain
$$\Ext^1(Q,L) \rightarrow \HH^1(L^2) \rightarrow \Ext^1(F,L) 
\rightarrow \Ext^2(Q,L) \rightarrow \HH^2(L^2) \rightarrow \Ext^2(F,L)
\rightarrow 0.$$
The map appearing in $\iota)$ can be described by the following three components
\begin{equation}\begin{matrix}\label{ims}
 \Ext^2(Q,L)  & \otimes &  \Hom(F,Q)& & \\
\uparrow& & \downarrow& \searrow & \\
 \Ext^1(F,L) & \otimes &   \Ext^1(F,F)   &\rightarrow & \Ext^2(F,L).\\
\uparrow & & \downarrow & \nearrow & \\
\Ext^1(F^{\vee\vee},L)  & \otimes &  \Ext^1(F, F^{\vee\vee})& 
\end{matrix}
\end{equation}
The diagram commutes, so $\Gamma \ext^1(F,F)$ kills 
$\im\Ext^1(F^{\vee\vee},L)$ and pairs with $\Ext^2(Q,L).$
We are interested in a neighbourhood 
of $[E]$, that is, the class of $E$ in $\Ext^1(F,L)$.
Consider the image of
$[E]$ inside  $\Ext^2(Q,L).$ 
Recall that the non-zero fibers of $Q$ are one-dimensional, because 
$Q$ is a quotient of $\mathcal{O}.$ 
Since $E$ is a vector bundle, it maps to non-zero values in the fibres of
$\Gamma \ext^2(Q,L)$ at all points in the supp($Q$).
Hence, by acting on $\im[E] \in \Ext^2(Q,L)$
with elements of $\Hom(Q,Q)$, we obtain
any element of $\Ext^2(Q,L).$
It follows that 
$\im(\Hom(F,Q) \otimes [E])= \Ext^2(Q,L) \otimes \im\Hom(F,Q))$ in $\Ext^2(F,L)$,
where $\im$ denotes the respective image by the map depicted in diagram\thinspace \ref{ims}.
Moreover, 
as $[E]$ varies, the zero set of the Kuranishi map projected to $\Ext^1(F, F^{\vee\vee})$
forms a vector bundle over its image with fibre Ker$( . \otimes [E]).$ 
In other words, near $[E]$ 
the quadratic part of the Kuranishi map comes entirely from 
the pairing 
$$\im\Ext(L^{-1},L)  \otimes  \Ext^1(L^{-1},L^{-1}) 
\rightarrow \im\Ext^2(L^{-1},L)$$
which is the same as the one for the case of a split bundle case, described in
Sec.\thinspace \ref{Kur-split}, 
and the same conclusions apply. 
Hence, we arrive at the same map as in the case $\iota)$ of Sec.\thinspace \ref{Kur-split}:
$$ \HH^1(\Sigma,L^2)\otimes \HH^1(\Sigma,\mathcal{O})  \rightarrow  \HH^2(\Sigma,L^2).$$

$\left.\iota\iota\right)$
Considering the map $\Ext^1(L,F) \otimes \Ext^1(F,L)
  \rightarrow \HH^2(\mathcal{O})$, we see that the target space 
$\HH^2(\mathcal{O})$ is fixed, so this component imposes the expected bounded number of obstructions, 
which are negligible for large $c_2$.
\vspace{3mm}

$\left.\iota\iota\iota\right)$
Finally we consider the map $$ \Ext^1(F,F) \otimes \Ext^1(L,F) 
  \rightarrow \mathrm{Ext}^2(L,F)= \HH^2(L^{-2}).$$
 We  have
$\HH^2(L^{-2})= \HH^1(\Sigma_1,L_1^{-2} ) \otimes \HH^1(\Sigma_2, 
L_2^{-2})$ which does not vanish. We now try to remove $T_o.$
The argument is analogous to the one in the Sec.\thinspace \ref{Kur-split}, namely this component
  gives obstructions to the Kuranishi map, and is defined on 
$T_o \otimes T_s$  where  $T_s$
contains deformations toward stable bundles. We want to remove 
$T_o$ from the the Kuranishi space at $E$.
The spaces $T_o$ and $T_s$ are determined by the following exact sequences
$$0 \rightarrow \Hom(L,Q) \rightarrow T_s \rightarrow \HH^1(L^{-2} )
\rightarrow 0$$
$$0 \rightarrow \HH^1(\mathcal{O}) \rightarrow T_o \rightarrow \Gamma\ext^1(F,F) 
\rightarrow 0.$$

\begin{claim*} Hom$(L,Q)$ gets killed under the pairing.
\end{claim*}
\begin{proof} First consider the map
$\Hom(L,Q) \otimes \HH^1(\mathcal{O}) \rightarrow \Ext^2(L,F).$
Given that $\HH^1(\mathcal{O}) = \Ext^1(L,L)$ the previous map factors
through $\Ext^1(L,Q)$ and the latter is zero because $\ext^1(L,Q)=0$ since 
$L$ is locally free and  $\HH^1(\Hom(L,Q))=0$ because $Q$ is supported at points.
There remains  the map 
$$\Hom(L,Q) \otimes \Gamma \ext^1(F,F) \rightarrow \Ext^2(L,F).$$
The fact that this map is zero follows from the following stronger statement.
\end{proof}

\begin{claim*} $T_s =\Gamma \ext^1(F,F) $ gets killed under the pairing.
\end{claim*}
\begin{proof} We want to show that the map 
$$\Ext^1(L,F) \otimes \Gamma \ext^1(F,F)  \rightarrow 
\Ext^2(L,F)$$  is zero. 
But $\Gamma \ext^1(F,F) \simeq \Gamma\shom(F,Q) = \Hom(F,Q)$ hence the domain  becomes
$\Ext^1(L,F) \otimes \Hom(F,Q)$ 
and the map factors through 
$\Ext^1(L,Q)$ which is zero, because $\ext^1(L,Q)=0$.
\end{proof}

The third part of the Kuranishi map in this case gets reduced to the following:
$$\HH^1(\mathcal{O}) \otimes \HH^1(L^{-2}) \rightarrow \Ext^2(L,F) = \HH^2(L^{-2}),$$
but this is the same as  part $\iota\iota\iota$ in the case when $E$ was an extension of line bundles.
The remaining part of the argument proceeds as in Sec.\thinspace \ref{Kur-split}, and once again we obtain the same bounds.

In conclusion, we have showed that each component of the Kuranishi map 
gives the same bounds in this case for the set of unstable points, as we had obtained in Sec.\thinspace \ref{Kur-split}, 
and accordingly, we obtain a result analogous to Prop.\thinspace \ref{split}.

\begin{proposition} \label{nonsplit} Assume $c_2 > 8g_1g_2$. Then, removing the set of unstable  bundles
 does not change the homology of the Kuranishi space 
$\mathbb K_{c_2}(\Sigma, E)$
in dimension smaller than $c_2$.
\end{proposition}

\begin{remark}
Note that our result is independent of the choice of polarization.
Fix a polarization $\omega=(\alpha,\beta)\in \HH^2(\Sigma,\mathbb{Z})$ 
choose $L$ 
which is destabilizing for $E$, then if 
  $L\ce L_1\boxtimes L_2\rightarrow \Sigma$ is a line bundle of type $(m, n)$, 
  we have  $$\mathrm{deg}(L)=\alpha m+\beta n\geq 0.$$
This conclusion holds whether we are in the setting of Sec.\thinspace \ref{Kur-split} or 
Sec.\thinspace \ref{Kur-nonsplit} for any choice of $\omega.$
\end{remark}

Note that in Sec.\thinspace \ref{Kur-nonsplit}
 we have studied the neighborhoods of
nonfiltrable bundles, which also includes the case of filtrable bundles when $Q= \emptyset$ occurs,
and the same bounds apply, since after all, the bounds coincide with those of Sec.\thinspace \ref{Kur-split}.

\section{The topology of the Kuranishi space}

Our main result here concerns a neighborhood of an unstable bundle $E$. 
Note that since stability is an open condition, our homological statement is trivially true for 
a stable bundle $E$.
Otherwise, assume that $E$ has 
 a destabilizing bundle of degree $(m,n)$.  Since here $c_1=0$ and $c_2=-2mn>0$ 
one of $m,n$ must be positive. We will assume that  $m$ is positive.

Observe that
$\HH_0({\mathbb K}_{c_2}
(\Sigma, E), {\mathbb K}^s_{c_2}(\Sigma, E))= \mathbb Z$
if the deformation space of $E$ contains no stable bundles, and $0$ otherwise, 
because both ${\mathbb K}_{c_2}
(\Sigma, E)$ and $ {\mathbb K}^s_{c_2}(\Sigma, E)$ are path connected.

\begin{theorem}\label{main} Let $\Sigma_1$ and $\Sigma_2$ be Pic-independent smooth projective curves.
Let $E$ be bundle on $\Sigma_1\times \Sigma_2$ with trivial determinant 
and with second Chern number  $c_2> 8g_1g_2$. 
Then, for $0<q<c_2$
we have 
$$\HH_q({\mathbb K}_{c_2}
(\Sigma, E), {\mathbb K}^s_{c_2}(\Sigma, E))=0.$$
\end{theorem}

\begin{proof} The Kuranishi 
map around $E$ was studied in Prop.\thinspace \ref{split} for the case of split bundles
and in Prop.\thinspace \ref{nonsplit} for the case of nonfiltrable bundles; in either
case it 
has 3 components which we labelled $\iota$, $\iota\iota$ and $\iota\iota\iota$.

First, consider the case when $\mathbb K_{c_2}(\Sigma,E)$ contains no stable bundles. Then 
the domains of the components $\iota\iota$ and $\iota\iota\iota$ are both zero, because $T_s=0$. 
So, there remains only component $\iota$ whose target 
$\HH^2(\Sigma,L^2) = \HH^1(\Sigma_1,L_1^2) \otimes \HH^1(\Sigma_2,L_2^2)$  vanishes 
whenever  $2m=deg(L^2) > 2g_1-2$.
Therefore,   if
$m > g_i-1$ the Kuranishi map is identically zero in this case, so that 
${\mathbb K}_{c_2}(\Sigma, E)$ is a vector space and the  homology statement follows immediately.

Second, consider the case when the deformation space of $E$ contains stable bundles. 
 We use the results of Sec.\thinspace \ref{Kur-split} and \ref{Kur-nonsplit} to conclude that 
the codimension of the set of unstable bundles   is large and grows with $c_2$. 
Observe that the Kuranishi  space is defined in a neighborhood of a singular  point $E$ 
 by at most $\HH^2(\End_0 E)$ equations.  
The codimension requires fine estimates, 
 which we carried out separately for each component of the Kuranishi map. In fact, we showed in 
 Prop.\thinspace \ref{split} and \ref{nonsplit} that the difference of codimension and number of local defining 
 equations is bounded below by $c_2$. 
 Then, an application of Eq.\thinspace (\ref{kirwan}),
combined with the estimate on the  codimension of the set of unstable bundles for each
 component of the Kuranishi map,
 shows that removing unstable points 
 does not change the  homology  of the Kuranishi deformation space  in  degrees $0<q<c_2$.
 So that ${\mathbb K}_{c_2}(\Sigma, E)$ and $ {\mathbb K}^s_{c_2}(\Sigma, E)$ have the same 
 homology in this range, 
 as we wished to show. 
\end{proof}

We finish with a comment about the bounds on $c_2$.
 For the cases when the deformation space contains stable bundles, 
 we have required	 $m > 2g_1-2$ for components $\iota$ and $\iota\iota$ to obtain
  $\HH^1(\Sigma_1, L_1^2)=0$,  and then for component $\iota\iota\iota$  we asked for 
 $m$  large enough to make $\HH^0(\Sigma_1, L_1^{-2})= 0,$
 which produces the same bound on $m$.
Note that the bound on $m$ needed in case the deformation space has only unstable bundles
is smaller, so it does not influence the result. 
Since $c_2=-2mn$ and the situation is symmetric in $g_1,g_2$, 
we could instead have required $-n > 2g_2-2$.
So, the optimal bound to consider is $c_2= -2mn >2(2g_1-2)(2g_2-2)$. 
Hence, it is enough to require $c_2 >8g_1g_2$.\\

\paragraph{\bf Acknowledgements}  
E. Ballico
 was partially supported by MIUR and GNSAGA of INdAM (Italy).
F. Rubilar was supported by 
Beca Doctorado Nacional Conicyt Folio 21170589.
E. Gasparim and F. Rubilar were also supported by the 
Vicerrector\'\i a de Investigaci\'on y Desarrollo Tecnol\'ogico (UCN Chile).
  B. Suzuki was supported by the ANID-FAPESP cooperation 2019/13204-0.

\appendix

\section{The unstable component}\label{unstable}

In this section we show that an unstable bundle may have an 
entire Zariski open neighborhood consisting of unstable bundles.
This is relevant to our topological estimates, because for such a bundle $E$
the subset $\mathbb K^s(\Sigma, E)$ is empty. 
Alternatively, we may rephrase this in the language of stacks, as 
follows. The moduli stack of rank 2 vector bundles 
on a projective surface may have  entire  components of large dimension
consisting only of unstable bundles. 
Hence, the stack of all bundles may be reducible 
for all (arbitrarily large) values of $c_2,$ even though for large
$c_2$ stable bundles will be contained in a definite irreducible component,
by  results of Gieseker--Li \cite{GL} and O'Grady \cite{OG}. For 
basic properties of stacks see \cite{Go}.

Let $X$ be a smooth and connected projective surface. Fix an ample line bundle $H$ on $X$, a line bundle $R$ on $X$ (it will give our $c_1(E)$)
and a line bundle $L$ on $X$ such
that $2 L\cdot H > R\cdot H$ and $h^0(K _X\otimes R\otimes (L^{\otimes 2})^\vee )=0$. Note that for any fixed $H, R$ these
conditions are satisfied if we take $L =\HH^{\otimes t}$ with $t$ sufficiently large.
 Fix $c_2\in \mathbb {Z}$ such that $c_2 \ge -L^2 +L\cdot R$. Let $A(R,L,c_2)$
be the set of all isomorphism classes of rank 2 bundles $E$ which fit in an exact sequence
\begin{equation}\label{eqa1}
0 \to L \to E \to \mathcal {I}_Q\otimes R\otimes L^\vee \to 0
\end{equation}
with $Q$ a finite subset of $X$ with $\sharp (Q) = c_2+L^2-L \cdot R$. 
Since  $2 L\cdot H > R\cdot H$, each $E\in A(R,L,c_2)$ is slope $H$-unstable and (\ref{eqa1}) is its unique destabilizing filtration. In particular
no two non-proportional extensions (\ref{eqa1}), not even when associated to different sets $Q, Q'$, give isomorphic bundles.

\begin{remark}\label{a1}
	Take $c_2= -L^2 +L\cdot R$. In this case $A(R,L,c_2)$ is the set of all isomorphism classes of extensions of $R\otimes L^\vee$ by $L$.
	If $h^1((L^{\otimes {2}})^\vee\otimes R) =0$, then $A(R,L,c_2) =\{L\oplus R\otimes L^\vee \}$. 
\end{remark}

\begin{remark}\label{a2}
	Assume $ c_2+L^2-R\cdot L > 0$, then for any $Q\subset X$ with $\sharp (Q) =c_2+L^2-L \cdot R$, 
	 the set of all locally free $E$ fitting into (\ref{eqa1}) is
	isomorphic to a non-empty Zariski open subset of a projective space of  positive dimension 
		because any non-empty finite subset $A$ of $X$ has the 
		Cayley--Bacharach property with respect to $R\otimes (L^{\otimes 2})^\vee$ by our choice
	of $L$. Since the set of all subsets of $X$ with cardinality $c_2+L^2-L \cdot R$ is a variety of dimension $2(c_2+L^2-L \cdot R)$
	and each $E\in A(R,L,c_2)$ fits in a unique extension (\ref{eqa1}), up to a scalar multiple, we get $\dim (A(R,L,c_2)) \ge 2(c_2+L^2-L \cdot R)$.
	 Note that $Q$ is supported on zero dimensional scheme. 
\end{remark}

Now we reverse the data, i.e. we assume only that $c_1(R)=c_1(E)$ and $c_2\in \mathbb {Z}$. We fix an integer $a>0$. Fix an integer $t>0$
such that  $h^0(K _X\otimes R\otimes (\HH^{\otimes 2t})^\vee )=0$ and $-t^2H^2 +tH\cdot R \ge c_2+a$. Take $L:= H^{\otimes t}$.
The set $A(R,L,c_2)$ has dimension at least $2a$. 

 \section{Some homological algebra}\label{extcalc}
 The results 
of this appendix are used in Sec.\thinspace \ref{Kur-nonsplit}.

\begin{lemma}\label{locglob} Given that
 $L$ is a line bundle and $Q$ is a zero dimensional 
sheaf such 
$\displaystyle 
0 \rightarrow F \rightarrow L^{-1} \rightarrow Q \rightarrow 0,$ 
there exists 
a canonical splitting
$$ \mbox{\em Ext}^1(F,F) = \Gamma (\ext^1(F,F))
\oplus \HH^1({\mathcal  O}).$$
\end{lemma}

\begin{proof} Since $Q$ is supported at points, 
$\mbox{\em Hom}(Q,L^{-1})= \mbox{\em Ext}^1(Q,L^{-1})=0$.
Applying $\shom(Q, .)$ to the short exact sequence 
\begin{equation}
0 \rightarrow F \rightarrow L^{-1} \rightarrow Q \rightarrow 0, \label{mq}
\end{equation}
we obtain
$$0  \rightarrow 
\shom(Q,Q) \rightarrow
 \ext^1(Q,F)
\rightarrow 0\quad \text{and}$$
$$0  \rightarrow 
\ext^1(Q,Q) \rightarrow
 \ext^2(Q,F) \rightarrow \ext^2(Q,L^{-1}) \rightarrow 
\ext^2(Q,Q) \rightarrow 0.$$

The first part implies 
$\Ext^1(Q,F) = \mbox {Hom}(Q,Q) =  {\Bbb C}^l.$

For the second part, 
the last map is an isomorphism, 
so the penultimate map is zero and it follows that
$\ext^2(Q,F)\simeq \ext^1(Q,Q).$
Consequently,
\begin{equation}\Ext^2(Q,F) = \Ext^1(Q,Q)= {\Bbb C}^{2l}.\label{2l}\end{equation}

Since $\ext^1(L^{-1},F)=\ext^2(L^{-1},F)= \shom(Q,F)=0$, applying $\shom(.,F)$ to the short exact sequence $(\ref{mq})$ 
we obtain
$$ 0 \rightarrow 
\ext^1(F,F) \rightarrow
 \ext^2(Q,F) \rightarrow 0, $$
and combining with the result of (\ref{2l})  gives
\begin{equation} \Ext^1(F,F) \simeq 
 \Ext^1(Q,Q) = {\Bbb C}^{2l}. \label{2l2}\end{equation}
On the other hand, given that 
$\ext^1(F,F) =
 \ext^2(Q,F) $ and 
$\shom(F,Q)=\ext^1(Q,Q) $ the diagram 
$$\begin{matrix}0 &&&&0\cr
\downarrow & & & &\downarrow & && \cr
\shom(F,Q) & \rightarrow & \ext^1(F,F) & \rightarrow &
\ext^1(F,L^{-1}) \cr
 || & &\downarrow & & || & & & \cr
\ext^1(Q,Q) & \simeq &\ext^2(Q,F) & \stackrel{0}{\rightarrow} &
\ext^2(Q,L^{-1}) &  \cr
\end{matrix}$$
gives
$$\shom(F,Q) \simeq \ext^1(F,F). $$
Thus, 
\begin{equation}\Hom(F,Q) = \Gamma \ext^1(F,F).\label{ge}\end{equation}

The results of (\ref{2l2}) and (\ref{ge}) are then plugged into the 
left lower corner of the following  diagram
$$\begin{matrix}
\Hom(Q,Q) & \rightarrow & \Ext^1(Q,F) & \rightarrow &
0 & \rightarrow & \Ext^1(Q,Q) \cr
|| & &\downarrow & &\downarrow & &\downarrow & & \cr
\Hom(L^{-1},Q) & \rightarrow & \Ext^1(L^{-1},F) & \rightarrow &
\Ext^1(L^{-1},L^{-1}) & \rightarrow & 0 \cr
\downarrow & &\downarrow & &\downarrow & &\downarrow & & \cr
\Hom(F,Q) & \rightarrow & \Ext^1(F,F) & \rightarrow &
\Ext^1(F,L^{-1}) & \rightarrow & \Ext^1(F,Q) \cr
 || & &\downarrow & &\downarrow & &\downarrow & & \cr
\Ext^1(Q,Q) & \simeq & \Ext^2(Q,F) & \rightarrow &
\Ext^2(Q,L^{-1}) & \simeq & \Ext^2(Q,Q) \cr
\end{matrix}$$
and a little diagram chase  together with (\ref{ge}) yields\\
$\Ext^1(F,F)= \Ext^1(L^{-1},L^{-1}) \oplus \Hom(F,Q) = 
\HH^1({\mathcal  O}) \oplus \Gamma (\ext^1(F,F)).$
\end{proof}

\end{document}